\documentclass{tac} 
\usepackage[centertags,sumlimits,intlimits,namelimits,reqno]{amsmath}
\usepackage{latexsym,amsfonts,amssymb,exscale,enumerate,breakcites}
\usepackage[colorlinks,linkcolor=darkblue,citecolor=darkblue,urlcolor=darkblue,breaklinks=true]{hyperref}
\usepackage{tikz}
\usetikzlibrary{backgrounds,circuits,circuits.ee.IEC,shapes,fit,matrix,automata,decorations.markings}
\usepackage[all,cmtip,2cell]{xy}
\UseTwocells

  \author{Brendan Fong}

  \thanks{I thank John Baez, Daniel Marsden, and an anonymous referee for
    careful, detailed feedback on drafts, Jamie Vicary for useful conversations,
    and the Clarendon Fund and Hertford College, Oxford, for their support.  
  }

  \address{Department of Computer Science \\
    University of Oxford \\
    United Kingdom OX1 3QD
  }

  \title{Decorated cospans}

  \copyrightyear{2015}

  \keywords{cospan, decorated cospan, hypergraph category, well-supported
    compact closed category, separable algebra, Frobenius algebra, Frobenius
    monoid}
  \amsclass{18C10, 18D10}
  \eaddress{brendan.fong@cs.ox.ac.uk}

  \definecolor{darkblue}{rgb}{0,0,0.7} 
  \newcommand{\define}[1]{{\bf \boldmath #1}}
  \newcommand{\R}{{\mathbb{R}}}
  \newcommand{\hooklongrightarrow}{\lhook\joinrel\longrightarrow}
  \newcommand{\linsub}{\operatorname{LinSub}}
  \newcommand{\lgraph}{\operatorname{Graph}}
  \newcommand{\res}{\operatorname{Res}}
  \mathrmdef{FinSet}
  \mathrmdef{Set}
  \mathrmdef{opp}

  \pgfdeclarelayer{edgelayer}
  \pgfdeclarelayer{nodelayer}
  \pgfsetlayers{edgelayer,nodelayer,main}

  \tikzset{font=\footnotesize}
  \tikzstyle{none}=[inner sep=0pt]
  \tikzstyle{connection}=[circuit symbol open,
    circuit symbol size=width .5 height .5,
    shape=circle ee,
  inner sep=0.25\pgflinewidth]
  \tikzset{->-/.style={decoration={
    markings,
    mark=at position .54 with {\arrow{latex}}},postaction={decorate}}}

\begin{document}   

\maketitle

\begin{abstract}
  Let $\mathcal C$ be a category with finite colimits, writing its coproduct
  $+$, and let $(\mathcal D, \otimes)$ be a braided monoidal category. We
  describe a method of producing a symmetric monoidal category from a lax
  braided monoidal functor $F\colon (\mathcal C,+) \to (\mathcal D, \otimes)$, and of
  producing a strong monoidal functor between such categories from a monoidal
  natural transformation between such functors. The objects of these categories,
  our so-called `decorated cospan categories', are simply the objects of
  $\mathcal C$, while the morphisms are pairs comprising a cospan $X \rightarrow
  N \leftarrow Y$ in $\mathcal C$ together with an element $1 \to FN$ in
  $\mathcal D$. Moreover, decorated cospan categories are hypergraph
  categories---each object is equipped with a special commutative Frobenius
  monoid---and their functors preserve this structure.
\end{abstract}

\section{Introduction}

There is a well-known way to compose cospans in a category with finite colimits:
given cospans
\[
  \begin{aligned}
    \xymatrix{
      & N \\
      X \ar[ur]^{i_X} && Y \ar[ul]_{o_Y}
    }
  \end{aligned}
  \qquad \vcenter{\xymatrix{\mbox{and}}} \qquad
  \begin{aligned}
    \xymatrix{
      & M \\
      Y \ar[ur]^{i_Y} && Z, \ar[ul]_{o_Z}
    }
  \end{aligned}
\]
we take the pushout over their shared foot $Y$ 
\[
  \xymatrix{
    && P \\
    & N \ar[ur]^j && M \ar[ul]_{j'} \\
    X \ar[ur]^{i_X} && Y \ar[ul]_{o_Y} \ar[ur]^{i_Y} && Z \ar[ul]_{o_Z}
  }
\]
to get a cospan from $X$ to $Z$. In many situations, however, we wish to compose
`decorated' cospans, where the apex of each cospan is equipped with some extra
structure. In this article we detail a method for composing such decorated
cospans. 

Beyond category theoretic interest, the motivation for such a method lies in
developing compositional accounts of semantics associated to topological
diagrams. While this has long been a technique associated with topological
quantum field theory, dating back to \cite{At}, it has most recently had
significant influence in the nascent field of categorical network theory, with
application to automata and computation \cite{KSW2, Sp}, electrical circuits
\cite{BF}, signal flow diagrams \cite{BSZ, BE}, Markov processes \cite{BP,
ASW}, and dynamical systems \cite{VSL}, among others. 

It has been recognised for some time that spans and cospans provide an intuitive
framework for composing network diagrams \cite{KSW}, and the material we develop
here is a variant on this theme. In the case of finite graphs, the intuition
reflected is this: given two graphs, we may construct a third by gluing chosen
vertices of the first with chosen vertices of the second. It is our goal in this
article to view this process as composition of morphisms in a category, in a way
that also facilitates the construction of a composition rule for any semantics
associated to the diagrams, and a functor between these two resulting
categories.

To see how this works, let us start with the following graph:
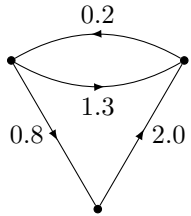
\begin{center}
  \begin{tikzpicture}[auto,scale=2.3]
    \node[circle,draw,inner sep=1pt,fill]         (A) at (0,0) {};
    \node[circle,draw,inner sep=1pt,fill]         (B) at (1,0) {};
    \node[circle,draw,inner sep=1pt,fill]         (C) at (0.5,-.86) {};
    \path (B) edge  [bend right,->-] node[above] {0.2} (A);
    \path (A) edge  [bend right,->-] node[below] {1.3} (B);
    \path (A) edge  [->-] node[left] {0.8} (C);
    \path (C) edge  [->-] node[right] {2.0} (B);
  \end{tikzpicture}
\end{center}
We shall work with labelled, directed graphs, as the additional data help
highlight the relationships between diagrams. Now, for this graph to be a
morphism, we must equip it with some notion of `input' and `output'. We do this by
marking vertices using functions from finite sets:
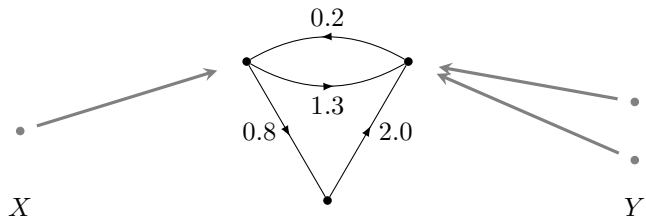
\begin{center}
  \begin{tikzpicture}[auto,scale=2.15]
    \node[circle,draw,inner sep=1pt,fill=gray,color=gray]         (x) at (-1.4,-.43) {};
    \node at (-1.4,-.9) {$X$};
    \node[circle,draw,inner sep=1pt,fill]         (A) at (0,0) {};
    \node[circle,draw,inner sep=1pt,fill]         (B) at (1,0) {};
    \node[circle,draw,inner sep=1pt,fill]         (C) at (0.5,-.86) {};
    \node[circle,draw,inner sep=1pt,fill=gray,color=gray]         (y1) at (2.4,-.25) {};
    \node[circle,draw,inner sep=1pt,fill=gray,color=gray]         (y2) at (2.4,-.61) {};
    \node at (2.4,-.9) {$Y$};
    \path (B) edge  [bend right,->-] node[above] {0.2} (A);
    \path (A) edge  [bend right,->-] node[below] {1.3} (B);
    \path (A) edge  [->-] node[left] {0.8} (C);
    \path (C) edge  [->-] node[right] {2.0} (B);
    \path[color=gray, very thick, shorten >=10pt, shorten <=5pt, ->, >=stealth] (x) edge (A);
    \path[color=gray, very thick, shorten >=10pt, shorten <=5pt, ->, >=stealth] (y1) edge (B);
    \path[color=gray, very thick, shorten >=10pt, shorten <=5pt, ->, >=stealth] (y2) edge (B);
  \end{tikzpicture}
\end{center}
Let $N$ be the set of vertices of the graph. Here the finite sets $X$, $Y$, and
$N$ comprise one, two, and three elements respectively, drawn as points, and the
values of the functions $X \to N$ and $Y \to N$ are indicated by the grey
arrows. This forms a cospan in the category of finite sets, one with the set at
the apex decorated by our given graph.

Given another such decorated cospan with input set equal to the output of the
above cospan
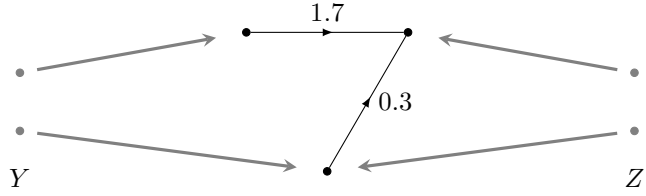
\begin{center}
  \begin{tikzpicture}[auto,scale=2.15]
    \node[circle,draw,inner sep=1pt,fill=gray,color=gray]         (y1) at (-1.4,-.25) {};
    \node[circle,draw,inner sep=1pt,fill=gray,color=gray]         (y2) at (-1.4,-.61) {};
    \node at (-1.4,-.9) {$Y$};
    \node[circle,draw,inner sep=1pt,fill]         (A) at (0,0) {};
    \node[circle,draw,inner sep=1pt,fill]         (B) at (1,0) {};
    \node[circle,draw,inner sep=1pt,fill]         (C) at (0.5,-.86) {};
    \node[circle,draw,inner sep=1pt,fill=gray,color=gray]         (z1) at (2.4,-.25) {};
    \node[circle,draw,inner sep=1pt,fill=gray,color=gray]         (z2) at (2.4,-.61) {};
    \node at (2.4,-.9) {$Z$};
    \path (A) edge  [->-] node[above] {1.7} (B);
    \path (C) edge  [->-] node[right] {0.3} (B);
    \path[color=gray, very thick, shorten >=10pt, shorten <=5pt, ->, >=stealth] (y1) edge (A);
    \path[color=gray, very thick, shorten >=10pt, shorten <=5pt, ->, >=stealth] (y2)
    edge (C);
    \path[color=gray, very thick, shorten >=10pt, shorten <=5pt, ->, >=stealth] (z1) edge (B);
    \path[color=gray, very thick, shorten >=10pt, shorten <=5pt, ->, >=stealth] (z2) edge (C);
  \end{tikzpicture}
\end{center}
composition involves gluing the graphs along the identifications
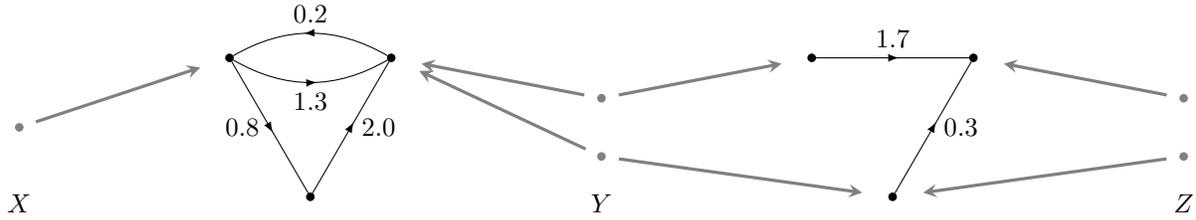
\begin{center}
  \begin{tikzpicture}[auto,scale=2.15]
    \node[circle,draw,inner sep=1pt,fill=gray,color=gray]         (x) at (-1.3,-.43) {};
    \node at (-1.3,-.9) {$X$};
    \node[circle,draw,inner sep=1pt,fill]         (A) at (0,0) {};
    \node[circle,draw,inner sep=1pt,fill]         (B) at (1,0) {};
    \node[circle,draw,inner sep=1pt,fill]         (C) at (0.5,-.86) {};
    \node[circle,draw,inner sep=1pt,fill=gray,color=gray]         (y1) at (2.3,-.25) {};
    \node[circle,draw,inner sep=1pt,fill=gray,color=gray]         (y2) at (2.3,-.61) {};
    \node at (2.3,-.9) {$Y$};
    \path (B) edge  [bend right,->-] node[above] {0.2} (A);
    \path (A) edge  [bend right,->-] node[below] {1.3} (B);
    \path (A) edge  [->-] node[left] {0.8} (C);
    \path (C) edge  [->-] node[right] {2.0} (B);
    \path[color=gray, very thick, shorten >=10pt, shorten <=5pt, ->, >=stealth] (x) edge (A);
    \path[color=gray, very thick, shorten >=10pt, shorten <=5pt, ->, >=stealth] (y1) edge (B);
    \path[color=gray, very thick, shorten >=10pt, shorten <=5pt, ->, >=stealth] (y2) edge (B);
    \node[circle,draw,inner sep=1pt,fill]         (A') at (3.6,0) {};
    \node[circle,draw,inner sep=1pt,fill]         (B') at (4.6,0) {};
    \node[circle,draw,inner sep=1pt,fill]         (C') at (4.1,-.86) {};
    \node[circle,draw,inner sep=1pt,fill=gray,color=gray]         (z1) at (5.9,-.25) {};
    \node[circle,draw,inner sep=1pt,fill=gray,color=gray]         (z2) at (5.9,-.61) {};
    \node at (5.9,-.9) {$Z$};
    \path (A') edge  [->-] node[above] {1.7} (B');
    \path (C') edge  [->-] node[right] {0.3} (B');
    \path[color=gray, very thick, shorten >=10pt, shorten <=5pt, ->, >=stealth] (y1) edge (A');
    \path[color=gray, very thick, shorten >=10pt, shorten <=5pt, ->, >=stealth] (y2)
    edge (C');
    \path[color=gray, very thick, shorten >=10pt, shorten <=5pt, ->, >=stealth] (z1) edge (B');
    \path[color=gray, very thick, shorten >=10pt, shorten <=5pt, ->, >=stealth] (z2) edge (C');
  \end{tikzpicture}
\end{center}
specified by the shared foot of the two cospans. This results in the decorated
cospan
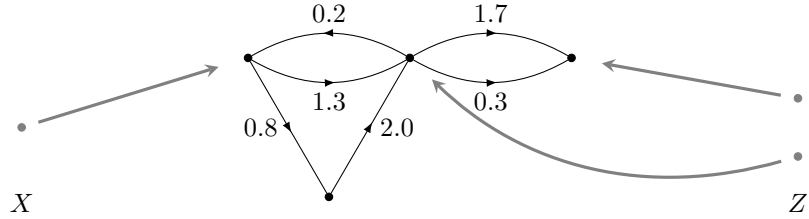
\begin{center}
  \begin{tikzpicture}[auto,scale=2.15]
    \node[circle,draw,inner sep=1pt,fill=gray,color=gray]         (x) at (-1.4,-.43) {};
    \node at (-1.4,-.9) {$X$};
    \node[circle,draw,inner sep=1pt,fill]         (A) at (0,0) {};
    \node[circle,draw,inner sep=1pt,fill]         (B) at (1,0) {};
    \node[circle,draw,inner sep=1pt,fill]         (C) at (0.5,-.86) {};
    \node[circle,draw,inner sep=1pt,fill]         (D) at (2,0) {};
    \node[circle,draw,inner sep=1pt,fill=gray,color=gray]         (z1) at (3.4,-.25) {};
    \node[circle,draw,inner sep=1pt,fill=gray,color=gray]         (z2) at (3.4,-.61) {};
    \node at (3.4,-.9) {$Z$};
    \path (B) edge  [bend right,->-] node[above] {0.2} (A);
    \path (A) edge  [bend right,->-] node[below] {1.3} (B);
    \path (A) edge  [->-] node[left] {0.8} (C);
    \path (C) edge  [->-] node[right] {2.0} (B);
    \path (B) edge  [bend left,->-] node[above] {1.7} (D);
    \path (B) edge  [bend right,->-] node[below] {0.3} (D);
    \path[color=gray, very thick, shorten >=10pt, shorten <=5pt, ->, >=stealth] (x) edge (A);
    \path[color=gray, very thick, shorten >=10pt, shorten <=5pt, ->, >=stealth] (z1)
    edge (D);
    \path[bend left, color=gray, very thick, shorten >=10pt, shorten <=5pt, ->, >=stealth] (z2)
    edge (B);
  \end{tikzpicture}
\end{center}
The decorated cospan framework generalises this intuitive construction.

More precisely: fix a set $L$. Then given a finite set $N$, we may talk of the
collection of finite $L$-labelled directed multigraphs, to us just $L$-graphs
or simply graphs, that have $N$ as their set of vertices. Write such a graph
$(N,E,s,t,r)$, where $E$ is a finite set of edges, $s\colon E \to N$ and $t\colon E \to N$
are functions giving the source and target of each edge respectively, and $r\colon  E
\to L$ equips each edge with a label from the set $L$.  Next, given a function
$f\colon N \to M$, we may define a function from graphs on $N$ to graphs on $M$
mapping $(N,E,s,t,r)$ to $(M,E,f \circ s,f \circ t, r)$.  After dealing
appropriately with size issues, this gives a lax monoidal functor from
$(\FinSet,+)$ to $(\Set,\times)$.\footnote{Here $(\FinSet,+)$ is the monoidal
  category of finite sets and functions with disjoint union as monoidal
  product, and $(\Set,\times)$ is the category of sets and functions with
  cartesian product as monoidal product. One might ensure the collection of
  graphs forms a set in a number of ways. One such method is as follows: the
  categories of finite sets and finite graphs are essentially small; replace
  them with equivalent small categories. We then constrain the graphs
  $(N,E,s,t,r)$ to be drawn only from the objects of our small category of finite
  graphs.}  

Now, taking any lax monoidal functor $(F,\varphi)\colon  (\mathcal C,+) \to (\mathcal
D,\otimes)$ with $\mathcal C$ having finite colimits and coproduct written $+$,
the decorated cospan category associated to $F$ has as objects the objects of
$\mathcal C$, and as morphisms pairs comprising a cospan in $\mathcal C$
together with some morphism $1 \to FN$, where $1$ is the unit in $(\mathcal
D,\otimes)$ and $N$ is the apex of the cospan. In the case of our graph
functor, this additional data is equivalent to equipping the apex $N$ of the
cospan with a graph. We thus think of our morphisms as having two distinct
parts: an instance of our chosen structure on the apex, and a cospan describing
interfaces to this structure. Our first theorem says that when $(\mathcal
D,\otimes)$ is braided monoidal and $(F,\varphi)$ lax braided monoidal, we may
further give this data a composition rule and monoidal product such that the
resulting `decorated cospan category' is symmetric monoidal with a special
commutative Frobenius monoid on each object.  

Suppose now we have two such lax monoidal functors; we then have two such
decorated cospan categories. Our second theorem is that, given also a monoidal
natural transformation between these functors, we may construct a strict
monoidal functor between their corresponding decorated cospan categories.  These
natural transformations can often be specified by some semantics associated to
some type of topological diagram. A trivial case of such is assigning to a
finite graph its number of vertices, but richer examples abound, including
assigning to a directed graph with edges labelled by rates its depicted Markov
process, or assigning to an electrical circuit diagram the current--voltage
relationship such a circuit would impose.

An advantage of the decorated cospan framework is that the resulting categories
are hypergraph categories, and the resulting functors respect this structure.
As dagger compact categories, hypergraph categories themselves have a rich
diagrammatic nature \cite{Se}, and in cases when our decorated cospan categories
are inspired by diagrammatic applications, the hypergraph structure provides
language to describe natural operations on our diagrams, such as juxtaposing,
rotating, and reflecting them.

\subsection{Outline.}
The structure of this paper is straightforward: in the following section we
review some basic background material, which then allows us to give the
constructions of decorated cospan categories and their functors in Sections
\ref{sec:dcc} and \ref{sec:dcf} respectively. We then explicate these
definitions through some examples in Section \ref{sec:ex}. For completeness, we
supply further details of our proofs in the Appendix \ref{sec:proofs}.

\subsection{Notation.}
We shall assume the following standard names for certain distinguished objects
and morphisms, only disambiguating the symbols with subscripts when we judge
that the extra clarity is worth the clutter. We write: 
\begin{itemize} 
  \item $1$ for both identity morphisms and monoidal units, leaving context to
    determine which one we mean.
  \item $\lambda$, $\rho$, $a$, and $\sigma$ for respectively the left unitor, right unitor,
    associator, and, if present, braiding, in a monoidal category.
  \item $\varnothing$ for the initial object in a category.
  \item $!$ for the unique map from the initial object to a given object.
\end{itemize}

\section{Background}
\subsection{Cospan categories.}
Recall that a \define{cospan} from $X$ to $Y$ in a category $\mathcal C$ is an
object $N$ in $\mathcal C$ with a pair of morphisms $(i\colon  X \to N$, $o\colon  Y \to
N$):
\[
  \xymatrix{
    & N \\
    X \ar[ur]^{i} && Y. \ar[ul]_{o}
  }
\]
We shall refer to $X$ and $Y$ as the \define{feet}, and $N$ as the
\define{apex} of the cospan.  Cospans may be composed using the pushout from
the common foot, when such a pushout exists: given cospans $X
\stackrel{i_X}{\longrightarrow} N \stackrel{o_Y}{\longleftarrow} Y$ from $X$ to
$Y$ and $Y \stackrel{i_Y}{\longrightarrow} M \stackrel{o_Z}{\longleftarrow} Z$
from $Y$ to $Z$, their composite cospan is $X \stackrel{j \circ
i_X}{\longrightarrow} P \stackrel{j'\circ i_Z}{\longleftarrow} Z$, where $P$,
$(j\colon  N \to P)$, and $(j'\colon  M \to P)$ form the top half of the pushout square
\[
  \xymatrix{
    && P \\
    & N \ar[ur]^j && M \ar[ul]_{j'} \\
    X \ar[ur]^{i_X} && Y \ar[ul]_{o_Y} \ar[ur]^{i_Y} && Z. \ar[ul]_{o_Z}
  }
\]
A \define{map of cospans} is a morphism $n\colon  N \to N'$ in $\mathcal C$ between
the apices of two cospans $X \stackrel{i}{\longrightarrow} N
\stackrel{o}{\longleftarrow} Y$ and $X \stackrel{i'}{\longrightarrow} N'
\stackrel{o'}{\longleftarrow} Y$ with the same feet, such that both triangles 
\[
  \xymatrix{
    & N \ar[dd]^n  \\
    X \ar[ur]^{i} \ar[dr]_{i'} && Y \ar[ul]_{o} \ar[dl]^{o'}\\
    & N'
  }
\]
commute. Given a category $\mathcal C$ with pushouts, we may define a category
$\mathrm{Cospan}(\mathcal C)$ with objects the objects of $\mathcal C$ and
morphisms isomorphism classes of cospans \cite{Be}. We will often abuse our
terminology and refer to cospans themselves as morphisms in some cospan
category $\mathrm{Cospan}(\mathcal C)$; we of course refer instead to the
isomorphism class of the said cospan.

\subsection{Hypergraph categories.}

A \define{Frobenius monoid} $(X,\mu,\delta,\eta,\epsilon)$ in a monoidal
category $(\mathcal C, \otimes)$ is an object $X$ together with monoid $(X,\mu,
\eta)$ and comonoid $(X,\delta,\epsilon)$ structures such that
\[
  (1 \otimes \mu) \circ (\delta \otimes 1) = \delta \circ \mu = (\mu \otimes 1)
  \circ (1 \otimes \delta)\colon  X \otimes X \longrightarrow X \otimes X.
\]
A Frobenius monoid is further called \define{special} if 
\[
  \mu \circ \delta = 1\colon  X \longrightarrow X,
\]
and further called \define{commutative} if the ambient monoidal category is symmetric
and the monoid and comonoid structures that comprise the Frobenius monoid are
commutative and cocommutative respectively. Note that for Frobenius monoids
commutativity of the monoid structure implies cocommutativity of the comoniod
structure, and vice versa, so the use of the term `commutativity' for both the
Frobenius monoid and the constituent monoid is not ambiguous.

A \define{hypergraph category} is a
symmetric monoidal category in which each object is equipped with a special
commutative Frobenius structure $(X,\mu_X,\delta_X,\eta_X,\epsilon_X)$ such that 
\[
  \begin{array}{cc}
    \mu_{X\otimes Y} = (\mu_X \otimes \mu_Y)\circ(1_X \otimes \sigma_{YX}\otimes
    1_Y) \qquad&
    \eta_{X\otimes Y} = \eta_X \otimes \eta_Y \\
    \delta_{X\otimes Y} = (1_X \otimes \sigma_{XY}\otimes 1_Y)\circ(\delta_X
    \otimes \delta_Y) \qquad&
    \epsilon_{X\otimes Y} = \epsilon_X \otimes \epsilon_Y.
  \end{array}
\]
A functor $(F,\varphi)$ of hypergraph categories, or \define{hypergraph
functor}, is a strong symmetric monoidal functor $(F,\varphi)$ that preserves
the hypergraph structure. More precisely, the latter condition means that given
an object $X$, the special commutative Frobenius structure on $FX$ must be 
\[
  (FX,\enspace F\mu_X \circ \varphi_{X,X},\enspace  \varphi^{-1} \circ F\delta_X,\enspace  F\eta_X \circ
\varphi_1,\enspace  \varphi_1 \circ \epsilon_X).
\]

This terminology was introduced recently \cite{Ki}, in reference to the fact
that these special commutative Frobenius monoids provide precisely the structure
required to draw graphs with `hyperedges': wires connecting any number of
inputs to any number of outputs. Commutative special Frobenius monoids are also
known as commutative separable algebras \cite{RSW}, and hypergraph categories as
well-supported compact closed categories \cite{Ca}.

Note that if an object $X$ is equipped with a Frobenius monoid structure then
the maps $\epsilon \circ \mu\colon  X \otimes X \longrightarrow 1$ and $\delta \circ
\eta\colon  1 \longrightarrow X \otimes X$ obey 
\[
  \big(1 \otimes (\epsilon \circ \mu)\big) \circ \big((\delta \circ \eta)
  \otimes 1\big) = 1_X = \big((\epsilon \circ \mu) \otimes 1\big) \circ \big(1
  \otimes (\delta \circ \eta)\big)\colon X \longrightarrow X.
\]
Thus if an object carries a Frobenius monoid it is also self-dual, and any
hypergraph category is a fortiori self-dual compact closed. Mapping each
morphism $f\colon  X \to Y$ to its dual morphism
\[
  \big((\epsilon_Y \circ \mu_Y) \otimes 1_X\big) \circ \big( 1_Y \otimes f
  \otimes 1_X \big) \circ \big(1_Y \otimes (\delta_X \circ \eta_X)\big)\colon  Y
  \longrightarrow X
\]
further equips each hypergraph category with a so-called dagger functor---an
involutive contravariant endofunctor that is the identity on objects---such that
the category is a dagger compact category. Dagger compact categories were first
introduced in the context of categorical quantum mechanics \cite{AC}, under the
name strongly compact closed category, and have been demonstrated to be a key
structure in diagrammatic reasoning and the logic of quantum mechanics.

We shall see that every decorated cospan category is a hypergraph category, and
hence also a dagger compact category.

\begin{example}
  A central example of a hypergraph category is the category
  $\mathrm{Cospan(\mathcal C)}$ of cospans in any category $\mathcal C$ with
  finite colimits. We will later see that decorated cospan categories are a
  generalisation of such categories, and each inherits a hypergraph structure
  from such. 

  First, $\mathrm{Cospan(\mathcal C)}$ inherits a symmetric monoidal structure
  from $\mathcal C$. We call a subcategory $\mathcal C$ of a category $\mathcal
  D$ \define{wide} if $\mathcal C$ contains all objects of $\mathcal D$, and
  call a functor that is faithful and bijective-on-objects a \define{wide
  embedding}. Note then that we have a wide embedding
  \[
    \mathcal C \hooklongrightarrow \mathrm{Cospan(\mathcal C)}
  \]
  that takes each object of $\mathcal C$ to itself as an object of
  $\mathrm{Cospan(\mathcal C)}$, and each morphism $f\colon  X \to Y$ in $\mathcal C$
  to the cospan
  \[
    \xymatrix{
      & Y \\
      X \ar[ur]^{f} && Y, \ar@{=}[ul]
    }
  \]
  where the extended `equals' sign denotes an identity morphism. This allows us
  to view $\mathcal C$ as a wide subcategory of $\mathrm{Cospan(\mathcal C)}$.

  Now as $\mathcal C$ has finite colimits, it can be given a symmetric monoidal
  structure with the coproduct the monoidal product; we write this monoidal
  category $(\mathcal C,+)$, and write $\varnothing$ for the initial object, the
  monoidal unit of this category. Then $\mathrm{Cospan(\mathcal C)}$ inherits
  the same symmetric monoidal structure: since the monoidal product $+\colon \mathcal
  C \times \mathcal C \to \mathcal C$ is left adjoint to the diagram functor, it
  preserves colimits, and so extends to a functor $+\colon
  \mathrm{Cospan(\mathcal C)} \times \mathrm{Cospan(\mathcal C)} \to
  \mathrm{Cospan(\mathcal C)}$. The remainder of the monoidal structure is
  inherited because $\mathcal C$ is a wide subcategory of
  $\mathrm{Cospan(\mathcal C)}$.

  Next, the Frobenius structure comes from copairings of identity morphisms. We
  call cospans 
  \[
    \xymatrix{
      & N \\
      X \ar[ur]^{i} && Y \ar[ul]_{o}
    }
    \qquad \xymatrix@R=8pt{\\\textrm{and}} \qquad 
    \xymatrix{
      & N \\
      Y \ar[ur]^{o} && X \ar[ul]_{i}
    }
  \]
  that are reflections of each other \define{opposite} cospans. Given any object
  $X$ in $\mathcal C$, the copairing $[1_X,1_X]\colon  X + X \to X$ of two identity
  maps on $X$, together with the unique map $!\colon  \varnothing \to X$ from the
  initial object to $X$, define a monoid structure on $X$. Considering these
  maps as morphisms in $\mathrm{Cospan(\mathcal C)}$, we may take them together
  with their opposites to give a special commutative Frobenius structure on $X$.
  In this way we consider each category $\mathrm{Cospan(\mathcal C)}$ a
  hypergraph category.

  It is a simple computation to check that the resulting dagger functor simply
  takes a cospan $X \stackrel{i}{\longrightarrow} N \stackrel{o}{\longleftarrow}
  Y$ to its opposite cospan $Y \stackrel{o}{\longrightarrow} N
  \stackrel{i}{\longleftarrow} X$.
\end{example}

\section{Decorated cospan categories} \label{sec:dcc}

We now detail our central construction and state the main theorem.
\begin{definition} \label{def:fcospans}
  Let $\mathcal C$ be a category with finite colimits, and
  \[
    (F,\varphi)\colon  (\mathcal C,+) \longrightarrow (\mathcal D, \otimes)
  \]
  be a lax monoidal functor. We define a \define{decorated cospan}, or more
  precisely an $F$-decorated cospan, to be a pair 
  \[
    \left(
    \begin{aligned}
      \xymatrix{
	& N \\  
	X \ar[ur]^{i} && Y \ar[ul]_{o}
      }
    \end{aligned}
    ,
    \qquad
    \begin{aligned}
      \xymatrix{
	FN \\
	1 \ar[u]_{s}
      }
    \end{aligned}
    \right)
  \]
  comprising a cospan $X \stackrel{i}\rightarrow N \stackrel{o}\leftarrow Y$ in
  $\mathcal C$ together with an element $1 \stackrel{s}\rightarrow FN$ of
  the $F$-image $FN$ of the apex of the cospan. We shall call the element $1
  \stackrel{s}\rightarrow FN$ the \define{decoration} of the decorated
  cospan. A morphism of decorated cospans 
  \[
    n\colon  \big(X \stackrel{i_X}\longrightarrow N \stackrel{o_Y}\longleftarrow
    Y,\enspace 1 \stackrel{s}\longrightarrow FN\big) \longrightarrow \big(X
    \stackrel{i'_X}\longrightarrow N' \stackrel{o'_Y}\longleftarrow Y,\enspace 1
    \stackrel{s'}\longrightarrow FN'\big)
  \]
  is a morphism $n\colon  N \to N'$ of cospans such that $Fn \circ s = s'$.
\end{definition}

\begin{proposition}
  There is a category $F\mathrm{Cospan}$ of $F$-decorated cospans, with objects
  the objects of $\mathcal C$, and morphisms isomorphism classes of
  $F$-decorated cospans. On representatives of the isomorphism classes,
  composition in this category is given by pushout of cospans in $\mathcal C$
  \[
    \xymatrix{
      && N+_YM \\
      & N \ar[ur]^{j_N} && M \ar[ul]_{j_M} \\
      \quad X \quad \ar[ur]^{i_X} && Y \ar[ul]_{o_Y} \ar[ur]^{i_Y} && \quad Z
      \quad \ar[ul]_{o_Z}
    }
  \]
  paired with the composite
  \[
    1 \stackrel{\lambda^{-1}}\longrightarrow 1 \otimes 1 \stackrel{s \otimes
    t}\longrightarrow FN \otimes FM \stackrel{\varphi_{N,M}}\longrightarrow
    F(N+M) \stackrel{F[j_N,j_M]}\longrightarrow F(N+_YM)
  \]
  of the tensor product of the decorations with the $F$-image of the copairing
  of the pushout maps.
\end{proposition}

\begin{proof}
  The identity morphism on an object $X$ in a decorated cospan category is
  simply the identity cospan decorated as follows:
  \[
    \left(
    \begin{aligned}
      \xymatrix{
	& X \\  
	X \ar@{=}[ur] && X \ar@{=}[ul]
      }
    \end{aligned}
    ,
    \qquad
    \begin{aligned}
      \xymatrixrowsep{.5pc}
      \xymatrix{
	FX \\
	F\varnothing \ar[u]_{F!} \\
	1 \ar[u]_{\varphi_1}
      }
    \end{aligned}
    \right).
  \]
  We must check that the composition defined is well-defined on isomorphism
  classes, is associative, and, with the above identity maps, obeys the
  unitality axiom. These are straightforward, but lengthy, exercises in
  using the available colimits and monoidal structure to show that
  the relevant diagrams of decorations commute. The interested reader can find
  details in the appendix (\ref{app:welldefined}--\ref{app:identities}).
\end{proof}

\begin{remark}
While at first glance it might seem surprising that we can construct a
composition rule for decorations $s\colon  1\to FN$ and $t\colon  1 \to FM$ just from
monoidal structure, the copair $[j_N,j_M]\colon  N+M \to N+_YM$ of the pushout maps
contains the data necessary to compose them. Indeed, this is the key insight of the
decorated cospan construction. To wit, the coherence maps for the lax monoidal
functor allow us to construct an element of $F(N+M)$ from the monoidal product
$s \otimes t$ of the decorations, and we may then post-compose with $F[j_N,j_M]$ to
arrive at an element of $F(N+_YM)$. The map $[j_N,j_M]$ encodes the
identification of the image of $Y$ in $N$ with the image of the same in $M$, and
so describes merging the `overlap' of the two decorations.
\end{remark}

Our main theorem is that when \emph{braided} monoidal structure is present, the
category of decorated cospans is a hypergraph category, and moreover one into
which the category of `undecorated' cospans widely embeds.  This embedding
motivates the monoidal and hypergraph structures we put on $F\mathrm{Cospan}$.

\begin{theorem} \label{thm:fcospans}
  Let $\mathcal C$ be a category with finite colimits, $(\mathcal D, \otimes)$ a
  braided monoidal category, and $(F,\varphi)\colon  (\mathcal C,+) \to (\mathcal D,
  \otimes)$ be a lax braided monoidal functor. Then we may give
  $F\mathrm{Cospan}$ a symmetric monoidal and hypergraph structure such that
  there is a wide embedding of hypergraph categories
  \[
    \mathrm{Cospan(\mathcal C)} \hooklongrightarrow F\mathrm{Cospan}.
  \]
\end{theorem}

\begin{proof}
  Recall that the identity decorated cospan has apex decorated by $1
  \stackrel{\varphi_1}\longrightarrow F\varnothing \stackrel{F!}\longrightarrow
  FX$. Given any cospan $X \to N \leftarrow Y$, we call the decoration $1
  \stackrel{\varphi_1}\longrightarrow F\varnothing \stackrel{F!}\longrightarrow
  FN$ the \define{empty decoration} on $N$. We define a functor 
  \[
    \mathrm{Cospan(\mathcal C)} \hooklongrightarrow F\mathrm{Cospan}.
  \]
  mapping each object of $\mathrm{Cospan(\mathcal C)}$ to itself as an object
  of $F\mathrm{Cospan}$, and each cospan in $\mathcal C$ to the same cospan
  decorated with the empty decoration on its apex. As the composite of two
  empty-decorated cospans is again empty-decorated (see Appendix
  \ref{app:emptydecorations}), this defines a functor.

  We define the monoidal product of objects $X$ and $Y$ of $F\mathrm{Cospan}$ to
  be their coproduct $X+Y$ in $\mathcal C$, and define the monoidal product of
  decorated cospans $(X \stackrel{i_X}\longrightarrow N
  \stackrel{o_Y}\longleftarrow Y,\enspace 1 \stackrel{s}\longrightarrow FN)$ and
  $(X' \stackrel{i_{X'}}\longrightarrow N' \stackrel{o_{Y'}}\longleftarrow
  Y',\enspace 1 \stackrel{t}\longrightarrow FN')$ to be 
  \[
    \left(
    \begin{aligned}
      \xymatrix{
	& N+N' \\  
	X+X' \ar[ur]^{i_X+i_{X'}} && Y+Y' \ar[ul]_{o_Y+o_{Y'}}
      }
    \end{aligned}
    ,
    \qquad
    \begin{aligned}
      \xymatrixrowsep{.8pc}
      \xymatrix{
	F(N+N') \\
	FN \otimes FN' \ar[u]_{\varphi_{N,N'}}\\
	1 \otimes 1 \ar[u]_{s \otimes t} \\
	1 \ar[u]_{\lambda^{-1}}
      }
    \end{aligned}
    \right).
  \]
  Using the braiding in $\mathcal D$, we can show that this proposed monoidal
  product is functorial (Appendix \ref{app:monoidality}). Choosing associator,
  unitors, and braiding in $F\mathrm{Cospan}$ to be the images of those in
  $\mathrm{Cospan(\mathcal{C})}$, we have a symmetric monoidal category. These
  transformations remain natural transformations when viewed in the category of
  $F$-decorated cospans as they have empty decorations (see Appendix
  \ref{app:naturality}), and obey the required coherence laws as they are
  images of maps that obey these laws in $\mathrm{Cospan(\mathcal{C})}$. 

  Similarly, to arrive at the hypergraph structure on $F\mathrm{Cospan}$, we
  simply equip each object $X$ with the image of the special commutative
  Frobenius monoid specified by the hypergraph structure of
  $\mathrm{Cospan(\mathcal{C})}$. It is evident that this choice of structures
  implies the above wide embedding is a hypergraph functor.
\end{proof}

Note that if the monoidal unit in $(\mathcal D,\otimes)$ is the initial object,
then each object only has one possible decoration: the empty decoration. This
immediately implies the following corollary.
\begin{corollary}
  Let $1_{\mathcal C}\colon (\mathcal C,+) \to (\mathcal C,+)$ be the identity functor
  on a category $\mathcal C$ with finite colimits. Then
  $\mathrm{Cospan}(\mathcal C)$ and $1_{\mathcal C}\mathrm{Cospan}$ are
  isomorphic as hypergraph categories.
\end{corollary}

Thus we see that there is always a hypergraph functor between decorated cospan
categories $1_{\mathcal C}\mathrm{Cospan} \rightarrow F\mathrm{Cospan}$. This
provides an example of a more general way to construct hypergraph functors
between decorated cospan categories. We detail this in the next section.

\section{Functors between decorated cospan categories} \label{sec:dcf}

Decorated cospans provide a setting for formulating various operations that we
might wish to enact on the decorations, including the composition of these
decorations, both sequential and monoidal, as well as dagger, dualising, and
other operations afforded by the Frobenius structure. We now observe that these
operations are formulated in a systematic way, so that transformations of the
decorating structure---that is, monoidal transformations between the lax
monoidal functors defining decorated cospan categories---respect these
operations. 

\begin{theorem} \label{thm:decoratedfunctors}
  Let $\mathcal C$, $\mathcal C'$ be categories with finite colimits, abusing
  notation to write the coproduct in each category $+$, and $(\mathcal D,
  \otimes)$, $(\mathcal D',\boxtimes)$ be braided monoidal categories. Further let
  \[
    (F,\varphi)\colon  (\mathcal C,+) \longrightarrow (\mathcal D,\otimes)
  \]
  and
  \[
    (G,\gamma)\colon  (\mathcal C',+) \longrightarrow (\mathcal D',\boxtimes)
  \]
  be lax braided monoidal functors. This gives rise to decorated cospan
  categories $F\mathrm{Cospan}$ and $G\mathrm{Cospan}$. 

  Suppose then that we have a finite colimit-preserving functor $A\colon  \mathcal C
  \to \mathcal C'$ with accompanying natural isomorphism $\alpha\colon  A(-)+A(-)
  \Rightarrow A(-+-)$, a lax monoidal functor $(B,\beta)\colon  (\mathcal D, \otimes)
  \to (\mathcal D', \boxtimes)$, and a monoidal natural transformation $\theta\colon 
  (B \circ F, B\varphi\circ\beta) \Rightarrow (G \circ A, G\alpha\circ\gamma)$.
  This may be depicted by the diagram:
  \[
    \xymatrixcolsep{3pc}
    \xymatrixrowsep{3pc}
    \xymatrix{
      (\mathcal C,+) \ar^{(F,\varphi)}[r] \ar_{(A,\alpha)}[d] \drtwocell
      \omit{_\:\theta} & (\mathcal D,\otimes) \ar^{(B,\beta)}[d]  \\
      (\mathcal C',+) \ar_{(G,\gamma)}[r] & (\mathcal D',\boxtimes).
    }
  \]

  Then we may construct a hypergraph functor 
  \[
    (T, \tau)\colon  F\mathrm{Cospan} \longrightarrow G\mathrm{Cospan}
  \]
  mapping objects $X \in F\mathrm{Cospan}$ to $AX \in G\mathrm{Cospan}$, and
  morphisms 
  \[
    \left(
    \begin{aligned}
      \xymatrix{
	& N \\  
	X \ar[ur]^{i} && Y \ar[ul]_{o}
      }
    \end{aligned}
    ,
    \qquad
    \begin{aligned}
      \xymatrix{
	FN \\
	1_{\mathcal D} \ar[u]_{s}
      }
    \end{aligned}
    \right)
    \qquad
  to
  \qquad
    \left(
    \begin{aligned}
      \xymatrix{
	& AN \\  
	AX \ar[ur]^{Ai} && AY \ar[ul]_{Ao}
      }
    \end{aligned}
    ,
    \qquad
    \begin{aligned}
      \xymatrixrowsep{.8pc}
      \xymatrix{
	GAN \\
	BFN \ar[u]_{\theta_N}\\
	B1_{\mathcal D} \ar[u]_{Bs} \\
	1_{\mathcal D'} \ar[u]_{\beta_1}
      }
    \end{aligned}
    \right).
  \]
  Moreover, $(T,\tau)$ is a strict monoidal functor if and only if $(A,\alpha)$
  is.
\end{theorem}

\begin{proof}
  We must prove that $(T,\tau)$ is a functor, is strong symmetric monoidal, and
  that it preserves the special commutative Frobenius structure on each object.

  Checking the functoriality of $T$ is again an exercise in applying the
  properties of structure available---in this case the colimit-preserving nature
  of $A$ and the monoidality of $(\mathcal D,\boxtimes)$, $(B,\beta)$, and
  $\theta$---to show that the relevant diagrams of decorations commute. Again,
  the interested reader may find details in the appendix (\ref{app:functors}).
  
  The coherence maps of the functor are given by the coherence maps for the
  monoidal functor $A$, viewed now as cospans with the empty decoration. That
  is, we define the coherence maps $\tau$ to be the collection of isomorphisms
  \[
    \tau_1 = 
    \left(
    \begin{aligned}
      \xymatrix{
	& A\varnothing_{\mathcal C} \\  
	\varnothing_{\mathcal C'} \ar[ur]^{\alpha_1} && A\varnothing_{\mathcal
	C} \ar@{=}[ul]
      }
    \end{aligned}
    ,
    \qquad
    \begin{aligned}
      \xymatrixrowsep{.9pc}
      \xymatrix{
	GA\varnothing_{\mathcal C} \\
	G\varnothing_{\mathcal C'} \ar[u]_{G!} \\
	1_{\mathcal D'} \ar[u]_{\gamma_1}
      }
    \end{aligned}
    \right),
  \]
  \[
    \tau_{X,Y}=
    \left(
    \begin{aligned}
      \xymatrix{
	& A(X+Y) \\  
	AX+AY \ar[ur]^{\alpha_{X,Y}} && A(X+Y) \ar@{=}[ul]
      }
    \end{aligned}
    ,
    \qquad
    \begin{aligned}
      \xymatrixrowsep{.9pc}
      \xymatrix{
	GA(X+Y) \\
	G\varnothing_{\mathcal C'} \ar[u]_{G!} \\
	1_{\mathcal D} \ar[u]_{\gamma_1}
      }
    \end{aligned}
    \right),
  \]
  where $X$, $Y$ are objects of $G\mathrm{Cospan}$. As $(A,\alpha)$ is already
  strong symmetric monoidal and $\tau$ merely views these maps in $\mathcal C$
  as empty-decorated cospans in $G\mathrm{Cospan}$, $\tau$ is natural in $X$ and
  $Y$, and obeys the required coherence axioms for $(T,\tau)$ to also be strong
  symmetric monoidal (Appendix \ref{app:naturality2}). Moreover, as $A$ is
  coproduct-preserving and the Frobenius structures on $F\mathrm{Cospan}$ and
  $G\mathrm{Cospan}$ are built using various copairings of the identity map,
  $(T,\tau)$ preserves the hypergraph structure.

  Finally, it is straightforward to observe that the maps $\tau$ are identity
  maps if and only if the maps $\alpha$ are, so $(T,\tau)$ is a strict monoidal
  functor if and only $(A,\alpha)$ is.
\end{proof}

When the decorating structure comprises some notion of topological diagram, such
as a graph, these natural transformations $\theta$ might describe some semantic
interpretation of the decorating structure. In this setting the above theorem
constructs functorial semantics for the decorated cospan category of diagrams.
We conclude this paper with an example of this application of decorated cospans.

\section{Examples} \label{sec:ex}

In this final section we outline two constructions of decorated cospan
categories, based on labelled graphs and linear subspaces respectively, and a
functor between these two categories interpreting each graph as an electrical
circuit. We shall see that the decorated cospan framework allows us to take a
notion of closed system and construct a corresponding notion of open or
composable system, together with functorial semantics for these systems.

This electrical circuits example outlines the motivating application for the
decorated cospan construction; further details can be found in \cite{BF}.

\subsection{Labelled graphs.}

To begin we return to the example of this paper's introduction. 

Recall that a \define{$(0,\infty)$-graph} $(N,E,s,t,r)$ comprises a finite set
$N$ of vertices (or nodes), a finite set $E$ of edges, functions $s,t\colon  E \to N$
describing the source and target of each edge, and a function $r\colon  E \to
(0,\infty)$ labelling each edge. The decorated cospan framework allows us to
construct a category with, roughly speaking, these graphs as morphisms. More
precisely, our morphisms will consist of these graphs, together with subsets of
the nodes marked, with multiplicity, as `input' and `output' connection points.

As suggested in the introduction, pick small categories equivalent to the
categories of finite sets and $(0,\infty)$-graphs such that we may talk about
the set of all $(0,\infty)$-graphs on each finite set $N$.  Then we may consider
the functor
\[
  \lgraph\colon  (\FinSet,+) \longrightarrow (\Set,\times)
\]
taking a finite set $N$ to the set $\lgraph(N)$ of $(0,\infty)$-graphs
$(N,E,s,t,r)$ with set of nodes $N$. On
morphisms let it take a function $f\colon N \to M$ to the function that pushes
labelled graph structures on a set $N$ forward onto the set $M$:
\begin{align*}
  \lgraph(f)\colon  \lgraph(N) &\longrightarrow
  \lgraph(M); \\
  (N,E,s,t,r) &\longmapsto (M,E,f \circ s, f \circ t, r).
\end{align*}
As this map simply acts by post-composition, our map $\lgraph$ is indeed
functorial.

We then arrive at a lax braided monoidal functor $(\lgraph,\zeta)$ by equipping
this functor with the natural transformation 
\begin{align*}
  \zeta_{N,M}\colon  \lgraph(N) \times \lgraph(M)
  &\longrightarrow \lgraph(N+M); \\
  \big( (N,E,s,t,r), (M,F,s',t',r') \big) &\longmapsto
  \big(N+M,E+F,s+s',t+t',[r,r']\big),
\end{align*}
together with the unit map
\begin{align*}
  \zeta_1\colon  1=\{\bullet\} &\longrightarrow \lgraph(\varnothing); \\
  \bullet &\longmapsto
  (\varnothing,\varnothing,!,!,!),
\end{align*}
where we remind ourselves
that we write $[r,r']$ for the copairing of the functions $r$ and $r'$. The
naturality of this collection of morphisms, as well as the coherence laws for
lax braided monoidal functors, follow from the universal property of the coproduct.

Theorem \ref{thm:fcospans} thus allows us to construct a hypergraph category
$\mathrm{GraphCospan}$.  For an intuitive visual understanding of the morphisms
of this category and its composition rule, see this paper's introduction.

\subsection{Linear relations.}
Another example of a decorated cospan category arising from a functor $(\FinSet,+)
\to (\Set, \times)$ is closely related to the category of linear relations. Here
we decorate each cospan in $\Set$ with a linear subspace of $\R^N \oplus
(\R^N)^\ast$, the sum of the vector space generated by the apex $N$ over $\R$
and its vector space dual.

First let us recall some facts about relations. Let $R \subseteq X\times Y$ be
a relation; we write this also as $R\colon  X \to Y$. The opposite relation $R^\opp\colon 
Y\to X$, is the subset $R^\opp \subseteq Y\times X$ such that $(y,x) \in R^\opp$
if and only if $(x,y) \in R$. We say that the image of a subset $S \subseteq X$
under a relation $R\colon  X \to Y$ is the subset of all elements of the codomain $Y$
related by $R$ to an element of $S$. Note that if $X$ and $Y$ are vector spaces
and $S$ and $R$ are both linear subspaces, then the image $R(S)$ of $S$ under
$R$ is again a linear subspace.

Now any function $f\colon N \to M$ induces a linear map $f^\ast\colon  \R^M \to
\R^N$ by precomposition. This linear map $f^\ast$ itself induces a dual map
$f_\ast\colon (\R^N)^\ast \to (\R^M)^\ast$ by precomposition. Furthermore
$f^\ast$ has, as a linear relation $f^\ast \subseteq \R^M \oplus \R^N$, an
opposite linear relation $(f^\ast)^\opp\colon  \R^N \to \R^M$.  Define the
functor 
\[
  \linsub\colon  (\FinSet,+) \longrightarrow (\Set,\times)
\]
taking a finite set $N$ to set of linear subspaces of the vector space $\R^N
\oplus (\R^N)^\ast$, and taking a function $f\colon  N \to M$ to the function
$\linsub(N) \to \linsub(M)$ induced by the sum of these two relations:
\begin{align*}
  \linsub(f)\colon  \linsub(N) &\longrightarrow \linsub(M); \\
  L &\longmapsto \big((f^\ast)^\opp \oplus f_\ast\big)(L).
\end{align*}
The above operations on $f$ used in the construction of this map are functorial,
and so it is readily observed that $\linsub$ is indeed a functor.

It is moreover lax braided monoidal as the sum of linear subspace of $\R^N
\oplus (\R^N)^\ast$ and a linear subspace of $\R^M \oplus (\R^M)^\ast$ may be
viewed as a subspace of $\R^N \oplus (\R^N)^\ast \oplus \R^M \oplus (\R^M)^\ast
\cong \R^{N+M} \oplus (\R^{N+M})^\ast \cong \R^{M+N} \oplus (\R^{M+N})^\ast$,
and the empty subspace is a linear subspace of each $\R^N \oplus (\R^N)^\ast$.

We thus have a hypergraph category $\mathrm{LinSubCospan}$.

\subsection{Electrical circuits.} 
Electrical circuits and their diagrams are the motivating application for the
decorated cospan construction. Specialising to the case of networks of linear
resistors, we detail here how we may use the category $\mathrm{LinSubCospan}$ to
provide semantics for the morphisms of $\mathrm{GraphCospan}$ as diagrams of
networks of linear resistors.

Intuitively, after choosing a unit of resistance, say ohms ($\Omega$), each
$(0,\infty)$-graph can be viewed as a network of linear resistors, with the
$(0,\infty)$-graph of the introduction now more commonly depicted as
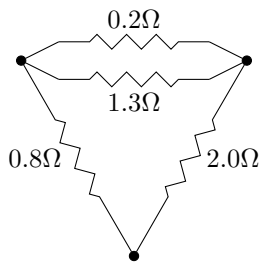
\begin{center}
  \begin{tikzpicture}[circuit ee IEC, set resistor graphic=var resistor IEC graphic]
    \node[contact]         (A) at (0,0) {};
    \node[contact]         (B) at (3,0) {};
    \node[contact]         (C) at (1.5,-2.6) {};
    \coordinate         (ua) at (.5,.25) {};
    \coordinate         (ub) at (2.5,.25) {};
    \coordinate         (la) at (.5,-.25) {};
    \coordinate         (lb) at (2.5,-.25) {};
    \path (A) edge (ua);
    \path (A) edge (la);
    \path (B) edge (ub);
    \path (B) edge (lb);
    \path (ua) edge  [resistor] node[label={[label distance=1pt]90:{$0.2\Omega$}}] {} (ub);
    \path (la) edge  [resistor] node[label={[label distance=1pt]270:{$1.3\Omega$}}] {} (lb);
    \path (A) edge  [resistor] node[label={[label distance=2pt]180:{$0.8\Omega$}}] {} (C);
    \path (C) edge  [resistor] node[label={[label distance=2pt]0:{$2.0\Omega$}}] {} (B);
  \end{tikzpicture}
\end{center}
$\mathrm{GraphCospan}$ may then be viewed as a category with morphisms
circuits of linear resistors equipped with chosen input and output terminals.

The suitability of this language is seen in the way the different categorical
structures of $\mathrm{GraphCospan}$ capture different operations that can be
performed with circuits. To wit, the sequential composition expresses the fact
that we can connect the outputs of one circuit to the inputs of the next, while
the monoidal composition models the placement of circuits side-by-side.
Furthermore, the symmetric monoidal structure allows us reorder input and output
wires, the compactness captures the interchangeability between input and
output terminals of circuits---that is, the fact that we can choose any input
terminal to our circuit and consider it instead as an output terminal, and vice
versa---and the Frobenius structure expresses the fact that we may wire any
node of the circuit to as many additional components as we like.

Moreover, Theorem \ref{thm:decoratedfunctors} provides semantics. Each node in a
network of resistors can be assigned an electric potential and a net current
outflow at that node, and so the set $N$ of vertices of a $(0,\infty)$-graph can
be seen as generating a space $\R^N \oplus (\R^N)^\ast$ of electrical states of
the network. We define a natural transformation 
\[
  \res\colon  \lgraph \Longrightarrow \linsub
\]
mapping each $(0,\infty)$-graph on $N$, viewed as a network of resistors, to the
linear subspace of $\R^N \oplus (\R^N)^\ast$ of electrical states permitted by
Ohm's law.\footnote{Note that these states need not obey Kirchhoff's current
law.} In detail, let $\psi \in \R^N$. We define the power $Q\colon  \R^N \to \R$
corresponding to a $(0,\infty)$-graph $(N,E,s,t,r)$ to be the function
\[
  Q(\psi) = \sum_{e \in E} \frac1{r(e)}
  \Big(\psi\big(t(e)\big)-\psi\big(s(e)\big)\Big)^2.
\]
Then the states of a network of resistors are given by a potential $\phi$ on the
nodes and the gradient of the power at this potential:
\begin{align*}
  \res_N\colon  \lgraph(N) &\longrightarrow \linsub(N)\\
  (N,E,s,t,r) &\longmapsto \{(\phi,\nabla Q_\phi) \mid \phi \in \R^N\}.
\end{align*}
This defines a monoidal natural transformation. Hence, by Theorem
\ref{thm:decoratedfunctors}, we obtain a hypergraph functor
$\mathrm{GraphCospan} \to \mathrm{LinSubCospan}$. 

The semantics provided by this functor match the standard interpretation of
networks of linear resistors. The maps of the Frobenius monoid take on the
interpretation of perfectly conductive wires, forcing the potentials at all
nodes they connect to be equal, and the sum of incoming currents to equal the
sum of outgoing currents---precisely the behaviour implied by Kirchhoff's laws.
More generally, let $(X \stackrel{i}{\rightarrow} N
\stackrel{o}{\leftarrow} Y, \, (N,E,s,t,r))$ be a morphism of
$\mathrm{GraphCospan}$, with $Q$ the power function corresponding to the graph
$\Gamma = (N,E,s,t,r)$. The image of this decorated cospan in
$\mathrm{LinSubCospan}$ is the decorated cospan $(X \stackrel{i}{\rightarrow} N
\stackrel{o}{\leftarrow} Y, \, \{(\phi,\nabla Q_\phi) \mid \phi \in \R^N\})$.
Then it is straightforward to check that the subspace 
\[
  \linsub[i,o]\big(\res_N(\Gamma)\big)\subseteq \R^{X+Y} \oplus (\R^{X+Y})^\ast
\]
is the subspace of electrical states on the terminals $X+Y$ such that currents
and potentials can be chosen across the network of resistors $(N,E,s,t,r)$ that
obey Ohm's and, on its interior, Kirchhoff's laws. In particular, after passing
to a subspace of the terminals in this way, composition in
$\mathrm{LinSubCospan}$ corresponds to enforcing Kirchhoff's laws on the shared
terminals of the two networks. A full exposition of this example can be found in
\cite{BF}.

\appendix
\section{Appendix: Proofs} \label{sec:proofs}

In this appendix we include the more technical aspects of the proofs that the
proposed constructions are well-defined and have the properties claimed. This
requires checking that a number of diagrams in decorated cospan categories
commute. In particular, here we check that the decorations agree; the required
properties of cospans themselves are well-established. While many of these
computations are on the routine side, we err on the side of more detail in the
hope that these details might be instructive in understanding more precisely
which aspects of monoidal functors imply given aspects of decorated cospan
categories.

\subsection{Representation independence for composition of isomorphism
classes of decorated cospans.} \label{app:welldefined}

\noindent

Let 
\[
  n\colon  \big(X \stackrel{i_X}\longrightarrow N \stackrel{o_Y}\longleftarrow Y,\enspace 1
\stackrel{s}\longrightarrow FN\big) \longrightarrow \big(X \stackrel{i'_X}\longrightarrow N'
\stackrel{o'_Y}\longleftarrow Y,\enspace 1 \stackrel{s'}\longrightarrow FN'\big)
\]
and
\[
  m\colon  \big(Y \stackrel{i_Y}\longrightarrow M \stackrel{o_Z}\longleftarrow Z,\enspace 1
\stackrel{t}\longrightarrow FM\big) \longrightarrow \big(Y \stackrel{i'_Y}\longrightarrow M'
\stackrel{o'_Z}\longleftarrow Z,\enspace 1 \stackrel{t'}\longrightarrow FM'\big)
\]
be isomorphisms of decorated cospans. We wish to show that the composite of the
decorated cospans on the left is isomorphic to the composite of the decorated
cospans on the right. As discussed, it is well-known that the composite cospans
are isomorphic, and it remains to us to check the decorations agree too. Let $p\colon 
N+_YM \to N'+_YM'$ be the isomorphism given by the universal property of the
pushout and the isomorphisms $n\colon N \to N'$ and $m\colon  M \to M'$. Then the two
decorations in question are given by the top and bottom rows of the following
diagram.
\[
  \xymatrixrowsep{1pc}
  \xymatrixcolsep{1pc}
  \xymatrix{
    &&&& FN \otimes FM \ar[dd]^{Fn \otimes Fm}_\sim
    \ar[rr]^{\varphi_{N,M}} && F(N+M) \ar[dd]^{F(n+m)}_\sim
    \ar[rr]^{F[j_N,j_M]} && F(N+_YM) \ar[dd]^{Fp}_\sim \\ 
    1 \ar[rr]^(.4){\lambda^{-1}} && 1\otimes 1 \ar[urr]^{s \otimes t}
    \ar[drr]_{s' \otimes t'} &
    \qquad\textrm{\tiny(I)} && \textrm{\tiny(F)} && \textrm{\tiny(C)}\\ 
    &&&& FN' \otimes FM' \ar[rr]_{\varphi_{N',M'}} && F(N'+M')
    \ar[rr]_{F[j_{N'},j_{M'}]} && F(N'+_YM')
  }
\]
The triangle (I) commutes as $n$ and $m$ are morphisms of decorated cospans and
$- \otimes -$ is functorial, (F) commutes by the monoidality of $F$, and (C)
commutes by properties of colimits in $\mathcal C$ and the functoriality of $F$.
This proves the claim.

\subsection{Associativity.} \label{app:associativity}

\noindent

Suppose we have morphisms
\[
  (X \stackrel{i_X}\longrightarrow N \stackrel{o_Y}\longleftarrow Y,\enspace 1
  \stackrel{s}\longrightarrow FN),
\]
\[
  (Y \stackrel{i_Y}\longrightarrow M \stackrel{o_Z}\longleftarrow Z,\enspace 1
  \stackrel{t}\longrightarrow FM), 
\]
\[
  (Z \stackrel{i_Z}\longrightarrow P \stackrel{o_W}\longleftarrow W,\enspace 1
  \stackrel{u}\longrightarrow FP).
\]
It is well-known that composition of isomorphism classes of cospans via
pushout of representatives is associative; this follows from the universal
properties of the relevant colimit. We must check that the pushforward of the
decorations is also an associative process. Write 
\[
  \tilde a\colon  (N+_YM)+_ZP \longrightarrow N+_Y(M+_ZP)
\]
for the unique isomorphism between the two pairwise pushouts constructions
from the above three cospans. Consider then the following diagram, with leftmost
column the decoration obtained by taking the composite of the first two morphisms
first, and the rightmost column the decoration obtained by taking the composite
of the last two morphisms first.
\[
  \xymatrixcolsep{-.6pc}
  \xymatrix{ 
    \scriptstyle F((N+_YM)+_ZP) \ar[rrrrrrrr]^{F\tilde a} &&&&&&&&
    \scriptstyle F(N+_Y(M+_ZP)) \\
    &&&& \textrm{\tiny(C)} \\
    \scriptstyle F((N+_YM)+P) \ar[uu]^{F[j_{N+_YM},j_P]} &&&&&&&& \scriptstyle
    F(N+(M+_ZP)) \ar[uu]_{F[j_N,j_{M+_ZP}]} \\
    && \scriptstyle F((N+M)+P) \ar[ull]_(.35)*+<6pt>_{\scriptstyle F([j_N,J_M]+1_P)}
    \ar[rrrr]^{Fa} &&&& \scriptstyle F(N+(M+P))
    \ar[urr]^(.35)*+<6pt>^{\scriptstyle F(1_N+[j_M,j_P])} \\
    & \textrm{\tiny(F2)} &&&&&& \textrm{\tiny(F3)} \\
    \scriptstyle F(N+_YM)\otimes FP \ar[uuu]^{\varphi_{N+_YM,P}} &&&& 
    \textrm{\tiny(F1)} &&&& \scriptstyle FN\otimes F(M+_ZP)
    \ar[uuu]_{\varphi_{N,M+_ZP}} \\
    && \scriptstyle F(N+M)\otimes FP \ar[ull]^(.6)*+<6pt>^{\scriptstyle
      F[j_N,j_M] \otimes 1_{FP}} \ar[uuu]_{\varphi_{N+M,P}} &&&& \scriptstyle
      FN \otimes F(M+P) \ar[urr]_(.6)*+<6pt>_{\scriptstyle 1_{FN} \otimes
    F[j_M,j_P]} \ar[uuu]^{\varphi_{N,M+P}} \\
    &&& \scriptstyle (FN \otimes FM) \otimes FP \ar[ul]^{\varphi_{N,M} \otimes
    1_{FP}} \ar[rr]^{a} && \scriptstyle FN \otimes (FM \otimes FP)
    \ar[ur]_{\phantom{1}1_{FN} \otimes \varphi_{M,P}} \\ 
    &&&& \textrm{\tiny(D2)} \\
    &&& \scriptstyle (1 \otimes 1) \otimes 1 \ar[uu]^{(s \otimes t) \otimes u}
    \ar[rr]^{a} && \scriptstyle 1 \otimes (1 \otimes 1) \ar[uu]_{s
    \otimes (t \otimes u)} \\
    &&&& \textrm{\tiny(D1)} \\
    &&&& \scriptstyle 1 \otimes 1 \ar[uul]^{\rho^{-1} \otimes 1} \ar[uur]_{1
      \otimes \lambda^{-1}} \\
      &&&& \scriptstyle 1 \ar[u]^{\lambda^{-1}}
  }
\]
This diagram commutes as (D1) is the triangle coherence equation for the
monoidal category $(\mathcal D,\otimes)$, (D2) is naturality for the associator
$a$, (F1) is the associativity condition for the monoidal functor $F$, (F2)
and (F3) commute by the naturality of $\varphi$, and (C) commutes as it is the
$F$-image of a hexagon describing the associativity of the pushout.  This
shows that the two decorations obtained by the two different orders of
composition of our three morphisms are equal up to the unique isomorphism
$\tilde a$ between the two different pushouts that may be obtained. Our
composition rule is hence associative.

\subsection{Identity morphisms.}  \label{app:identities}

\noindent 

We shall show that the claimed identity morphism on $Y$, the decorated cospan
\[
  (Y \stackrel{1_Y}\longrightarrow Y \stackrel{1_Y}\longleftarrow Y,\enspace 1
\stackrel{F!\circ \varphi_1}\longrightarrow FY),
\]
is an identity for composition
on the right; the case for composition on the left is similar. The cospan in
this pair is known to be the identity cospan in $\mathrm{Cospan}(\mathcal C)$.
We thus need to check that, given a morphism 
\[
  (X \stackrel{i}\longrightarrow N
\stackrel{o}\longleftarrow Y,\enspace 1 \stackrel{s}\longrightarrow FN),
\] 
the composite of the product $s \otimes (F! \circ \varphi_1)$ with the $F$-image
of the copairing $[1_N,i_Y]\colon  N+Y \to N$ of the pushout maps is again the same
element $s$; this composite being, by definition, the decoration of the
composite of the given morphism and the claimed identity map. This is shown by
the commutativity of the diagram below, with the path along the lower edge equal
to the aforementioned pushforward.
\[
  \xymatrixcolsep{1.1pc}
  \xymatrix{
    1 \ar[rrrrrrrrrr]^s \ar[ddrr]_{\rho_1^{-1}} &&&&&&&&&& FN \\
    &&&& \textrm{\tiny(D1)} &&&& \textrm{\tiny(F1)}\\
    && 1 \otimes 1 \ar[rr]^{s\otimes 1} \ar[ddrrrr]_{s \otimes (F! \circ
    \varphi_1)} && FN \otimes 1 \ar[uurrrrrr]^{\rho_{FN}} \ar[rr]^(.57){1_{FN}
  \otimes \varphi_1}
    && FN \otimes F\varnothing \ar[rr]^{\varphi_{N,\varnothing}} \ar[dd]_{1_{FN}
  \otimes F!} && F(N+\varnothing) \ar[uurr]^(.2){F\rho_N=F[1_N,!]}
      \ar[dd]_{F(1_N+!)} \quad \textrm{\tiny(C)}\\
      &&&&& \textrm{\tiny(D2)} && \textrm{\tiny(F2)}\\
      &&&&&& FN \otimes FY \ar[rr]_{\varphi_{N,M}} && F(N+Y)
      \ar[uuuurr]_{F[1_N,i_Y]}
  }
\]
This diagram commutes as each subdiagram commutes: (D1) commutes by the
naturality of $\rho$, (D2) by the functoriality of the monoidal product in
$\mathcal D$, (F1) by the unit axiom for the monoidal functor $F$, (F2) by the
naturality of $\varphi$, and (C) due to the properties of colimits in $\mathcal
C$ and the functoriality of $F$.

\subsection{Empty decorations.} \label{app:emptydecorations}

\noindent

We generalise the previous observation for identity morphisms. Let
\[
  (X \stackrel{i_X}\longrightarrow N
\stackrel{o_Y}\longleftarrow Y,\enspace 1 \stackrel{s}\longrightarrow FN),
\]
be a decorated cospan, and suppose we have an empty-decorated cospan 
\[
  (Y \stackrel{i_Y}\longrightarrow M \stackrel{o_Z}\longleftarrow Z,\enspace 1
  \stackrel{\varphi\circ F!}\longrightarrow FM).
\]
Here we show that the composite of these decorated cospans is 
\[
  \big(X \stackrel{j_N \circ i_X}\longrightarrow N+_YM \stackrel{j_M \circ
  o_Z}\longleftarrow Z,\enspace 1 \stackrel{Fj_N \circ s}\longrightarrow
  F(N+_YM)\big).
\]
In particular, the decoration on the composite is the decoration $s$ pushed
forward along the $F$-image of the map $j_N\colon N \to N+_YM$ to become an
$F$-decoration on $N+_YM$. We say that the empty decoration acts trivially on
other decorations. The analogous statement holds for composition with an
empty-decorated cospan on the left.

As is now familiar, a statement of this sort is proved by a large commutative
diagram:
\[
  \xymatrixcolsep{.63pc}
  \xymatrix{
    1 \ar[rrrrrr]^s \ar[ddrr]_{\rho_1^{-1}} &&&&&& FN \ar[rrrr]^{Fj_N} &&&& F(N+_YM) \\
    &&& \textrm{\tiny(D1)} &&& \textrm{\tiny(F1)} && \textrm{\tiny(C1)} \\
    && 1 \otimes 1 \ar[rr]^{s\otimes 1} \ar[ddrrrr]_{s \otimes (F! \circ
    \varphi_1)} && FN \otimes 1 \ar[uurr]^{\rho_{FN}} \ar[rr]^{1_{FN}
  \otimes \varphi_1}
    && FN \otimes F\varnothing \ar[rr]^{\varphi_{N,\varnothing}} \ar[dd]_{1_{FN}
  \otimes F!} && F(N+\varnothing) \ar[uull]_{F\rho_N=F[1_N,!]}
  \ar[uurr]^{F[j_N,!]}
      \ar[dd]_{F(1_N+!)} \quad \textrm{\tiny(C2)}\\
      &&&&& \textrm{\tiny(D2)} && \textrm{\tiny(F2)}\\
      &&&&&& FN \otimes FM \ar[rr]_{\varphi_{N,M}} && F(N+M)
      \ar[uuuurr]_{F[j_N,j_M]}
  }
\]
This subdiagrams in this diagram commute for the same reasons as their
corresponding regions in the previous diagram.

A consequence of the trivial action of the empty decoration is that the
composite of two empty-decorated cospans is again empty-decorated. This is
crucial for the functoriality of the embedding $\mathrm{Cospan(\mathcal{C})}
\hookrightarrow F\mathrm{Cospan}$.

\subsection{Functoriality of monoidal product.} \label{app:monoidality}

\noindent

Suppose we have decorated cospans
\[
  (X \stackrel{i_X}\longrightarrow N \stackrel{o_Y}\longleftarrow Y,\enspace 1
  \stackrel{s}\longrightarrow FN),
  \qquad
  (Y \stackrel{i_Y}\longrightarrow M \stackrel{o_Z}\longleftarrow Z,\enspace 1
  \stackrel{t}\longrightarrow FM), 
\]
\[
  (U \stackrel{i_U}\longrightarrow P \stackrel{o_V}\longleftarrow V,\enspace 1
  \stackrel{u}\longrightarrow FP),
  \qquad
  (V \stackrel{i_V}\longrightarrow Q \stackrel{o_W}\longleftarrow W,\enspace 1
  \stackrel{v}\longrightarrow FQ). 
\]
We must check the so-called interchange law: that the composite of the
column-wise monoidal products is equal to the monoidal product of the row-wise
composites.

Again, for the cospans we take this equality as familiar fact. Write 
\[
  b\colon  (N+P)+_{(Y+V)}(M+Q) \stackrel{\sim}{\longrightarrow} (N+_YM)+(P+_{V}Q).
\]
for the isomorphism between the two resulting representatives of the isomorphism
class of cospans. The two resulting decorations are then given by the leftmost 
and rightmost columns respectively of the diagram below.
\[
  \xymatrixcolsep{1.4pc}
  \xymatrix{ 
    \scriptstyle F((N+P)+_{(Y+V)}(M+Q)) \ar[rrrr]^{Fb}_{\sim} &&&&
    \scriptstyle F((N+_YM)+(P+_VQ)) \\
    & \textsc{\tiny(C)} \\
    \scriptstyle F((N+P)+(M+Q)) \ar[uu]^{F[j_{N+P},j_{M+Q}]} \ar@{.>}[uurrrr] &&&
    \textsc{\tiny(F2)} & \scriptstyle F(N+_YM)\otimes F(P+_VQ)
    \ar[uu]_{\varphi_{N+_YM,P+_VQ}} \\
    \\
    \scriptstyle F(N+P)\otimes F(M+Q) \ar[uu]^{\varphi_{N+P,M+Q}} &&&&
    \scriptstyle F(N+M)\otimes F(P+Q) \ar@{.>}[uullll]
    \ar[uu]_{F[j_N,j_M] \otimes F[j_P,j_Q]} \\
    && \textsc{\tiny(F1)} \\
    \scriptstyle (FN \otimes FP) \otimes (FM \otimes FQ) \ar[uu]^{\varphi_{N,P} \otimes
    \varphi_{M,Q}} \ar@{.>}[rrrr] &&&& \scriptstyle (FN \otimes FM) \otimes (FP \otimes FQ)
    \ar[uu]_{\varphi_{N,M} \otimes \varphi_{P,Q}} \\
    && \textsc{\tiny(D)} \\
    && \scriptstyle (1\otimes1)\otimes(1\otimes1) \ar[uull]^{(s\otimes u) \otimes (t \otimes
    v)} \ar[uurr]_{(s \otimes t) \otimes (u \otimes v)} \\
    && \scriptstyle 1 \ar[u]_{(\lambda^{-1}\otimes\lambda^{-1})\circ\lambda^{-1}}
  }
\]
These two decorations are related by the isomorphism $b$ as the diagram
commutes. We argue this more briefly than before, as the basic structure of
these arguments is now familiar to us. Briefly then, there exist dotted arrows
of the above types such that the subdiagram (D) commutes by the naturality of
the associators and braiding in $\mathcal D$, (F1) commutes by the coherence
diagrams for the braided monoidal functor $F$, (F2) commutes by the naturality of
the coherence map $\varphi$ for $F$, and (C) commutes by the properties of
colimits in $\mathcal C$ and the functoriality of $F$. 

Using now routine methods, it also is straightforward to show that the monoidal
product of identity decorated cospans on objects $X$ and $Y$ is the identity
decorated cospan on $X+Y$; for the decorations this amounts to the observation
that the monoidal product of empty decorations is again an empty decoration.

\subsection{Naturality of coherence maps.} \label{app:naturality}

\noindent

We consider the case of the left unitor; the naturality of the right unitor,
associator, and braiding follows similarly, using the relevant axiom where here
we use the left unitality axiom.

Given a decorated cospan $(X \stackrel{i}\longrightarrow N
\stackrel{o}\longleftarrow Y,\enspace 1 \stackrel{s}\longrightarrow FN)$, we
must show that the diagram of decorated cospans
\[
  \xymatrix{
    X+\varnothing \ar[r]^{i+1} \ar[d]_{\lambda_{\mathcal C}} & N+\varnothing & Y+\varnothing
    \ar[l]_{i+1} \ar[d]^{\lambda_{\mathcal C}} \\
    X \ar[r]_{i} & N & Y \ar[l]^{o} 
  }
\]
commutes, where the $\lambda_{\mathcal C}$ are the maps of the left unitor in $\mathcal C$
considered as empty-decorated cospans, and where the top cospan has decoration 
\[
  1 \stackrel{\lambda^{-1}_{\mathcal D}}{\longrightarrow} 1 \otimes 1 \stackrel{s \otimes
  \varphi_1}\longrightarrow FN \otimes F\varnothing
  \stackrel{\varphi_{N,\varnothing}}\longrightarrow F(N+\varnothing),
\]
and the lower cospan simply has decoration $1 \stackrel{s}{\longrightarrow} FN$. 

Now as the $\lambda$ are isomorphisms in $\mathcal C$ and have empty
decorations, by Appendix \ref{app:emptydecorations} the composite through the upper right
corner is isomorphic to the decorated cospan
\[
  \big(X+\varnothing \stackrel{i+1}\longrightarrow N+\varnothing
  \stackrel{(o+1)\circ \lambda^{-1}_{\mathcal C}}\longleftarrow Y,\enspace 1
  \stackrel{\varphi_{N,\varnothing} \circ (s \otimes \varphi_1) \circ
  \lambda^{-1}_{\mathcal D}}{\xrightarrow{\hspace*{2.5cm}}} F(N+\varnothing)\big),
\]
while the composite through the lower left corner is isomorphic to the decorated
cospan
\[
  (X+\varnothing \stackrel{[i,!]}\longrightarrow N
\stackrel{o}\longleftarrow Y,\enspace 1 \stackrel{s}\longrightarrow FN).
\]
Furthermore, $\lambda_{\mathcal C}\colon  N+\varnothing \rightarrow N$ gives an isomorphism
between these two cospans, and the naturality of the left unitor and the left
unitality axiom in $\mathcal D$ imply that this is in fact an isomorphism of
decorated cospans:
\[
  \xymatrix{
    1 \ar[rrr]^{s} \ar[d]_{\lambda^{-1}_{\mathcal D}} &&& FN \\
    1 \otimes 1 \ar[r]_{s \otimes 1} & FN \otimes 1 \ar[r]_{1 \otimes \varphi_1}
    \ar[urr]^{\lambda_{\mathcal D}} & FN \otimes F\varnothing
    \ar[r]_{\varphi_{N,\varnothing}} & F(N+\varnothing).
    \ar[u]_{F\lambda_{\mathcal C}}
  }
\]

\subsection{Functoriality of functors between decorated cospan categories.} \label{app:functors}

\noindent

Let 
\[
  (X \stackrel{i_X}\longrightarrow N \stackrel{o_Y}\longleftarrow Y, \enspace 1
  \stackrel{s}\longrightarrow FN)
  \qquad \mbox{and} \qquad
  (Y \stackrel{i_Y}\longrightarrow M \stackrel{o_Z}\longleftarrow Z, \enspace 1
  \stackrel{t}\longrightarrow FM), 
\]
be morphisms in $F\mathrm{Cospan}$. As the composition of the cospan part is by
pushout in $\mathcal C$ in both cases, and as $T$ acts as the colimit preserving
functor $A$ on these cospans, it is clear that $T$ preserves composition of
isomorphism classes of cospans. Write
\[
  c\colon  AX+_{AY}AZ \stackrel\sim\longrightarrow A(X+_YZ)
\]
for the isomorphism from the cospan obtained by composing the $A$-images of the
above two decorated cospans to the cospan obtained by taking the $A$-image of their
composite. To see that this extends to an isomorphism of decorated cospans,
observe that the decorations of these two cospans are given by the rightmost and
leftmost columns respectively in the following diagram:
\[
  \xymatrixcolsep{.5pc}
  \xymatrixrowsep{.9pc}
  \xymatrix{ 
    GA(N+_YM) &&&&&&&& \ar[llllllll]_{Gc}^{\sim} G(AN+_{AY}AM) \\
    &&&&& \textsc{\tiny(A)} \\
    BF(N+_YM) \ar[uu]^{\theta_{N+YM}} && \textsc{\tiny(T2)} && GA(N+M)
    \ar[uullll]_{GA[j_N,j_M]} &&&& G(AN+AM) \ar[uu]_{G[j_{AN},j_{AM}]}
    \ar[llll]^{\sim}_{G\alpha} \\
    \\
    BF(N+M) \ar[uu]^{BF[j_N,j_M]} \ar[uurrrr]_{\theta_{N,M}} &&&&
    \textsc{\tiny(T1)} &&&& GAN \boxtimes GAM \ar[uu]_{\gamma_{AN,AM}} \\
    \\
    B(FN \otimes FM) \ar[uu]^{B\varphi_{N,M}} &&&&&&&& BFN \boxtimes BFM
    \ar[uu]_{\theta_N \boxtimes \theta_M} \ar[llllllll]_{\beta_{FN,FM}} \\
    &&&& \textsc{\tiny(B2)} \\
    B(1_{\mathcal D} \otimes 1_{\mathcal D}) \ar[uu]^{B(s \otimes t)} &&&&&&&&
    B1_{\mathcal D} \boxtimes B1_{\mathcal D} \ar[uu]_{Bs \boxtimes Bt}
    \ar[llllllll]_{\beta_{1,1}} \\
    &&& \textsc{\tiny(B1)} &&&& \textsc{\tiny(D2)} \\
    B1_{\mathcal D} \ar[uu]^{\beta\lambda^{-1}_1} \ar[rrrr]^{\lambda^{-1}_{B1}}
    &&&& 1_{\mathcal D'} \boxtimes B1_{\mathcal D} \ar[uurrrr]^{\beta_1 \boxtimes 1}
    &&&& 1_{\mathcal D'} \boxtimes 1_{\mathcal D'} \ar[uu]^{\beta_1 \boxtimes \beta_1}
    \ar[llll]_{1 \boxtimes \beta_1} \\
    &&&& \textsc{\tiny(D1)} \\
    &&&& 1_{\mathcal D'} \ar[uullll]^{\beta_1} \ar[uurrrr]_{\lambda^{-1}_1}
  }
\]
From bottom to top, \textsc{(D1)} commutes by the naturality of $\lambda$,
\textsc{(D2)} by the functoriality of the monoidal product $-\boxtimes-$,
\textsc{(B1)} by the unit law for $(B,\beta)$, \textsc{(B2)} by the
naturality of $\beta$, \textsc{(T1)} by the monoidality of the natural
transformation $\theta$, \textsc{(T2)} by the naturality of $\theta$, and
\textsc{(A)} by the colimit preserving property of $A$ and the functoriality of
$G$.

We must also show that identity morphisms are mapped to identity morphisms. Let 
\[
  (X \stackrel{1_X}\longrightarrow X \stackrel{1_X}\longleftarrow X, \enspace 1
  \stackrel{F!\circ \varphi_1}\longrightarrow FX)
\]
be the identity morphism on some object $X$ in the category of $F$-decorated
cospans. Now this morphism has $T$-image
\[
  (AX \stackrel{1_{AX}}\longrightarrow AX \stackrel{1_{AX}}\longleftarrow AX, \enspace 1
  \stackrel{\theta_X \circ B(F!\circ \varphi_1) \circ
  \beta_1}{\xrightarrow{\hspace*{2.3cm}}} GAX).
\]
But we have the following diagram
\[
  \xymatrixcolsep{2pc}
  \xymatrixrowsep{.3pc}
  \xymatrix{ 
    &&&& BF\varnothing_{\mathcal C} \ar[ddrr]^{BF!}
    \ar[dddd]^{\theta_\varnothing} \\
    \\
    && B1_{\mathcal D} \ar[uurr]^{B\varphi_1} &&&& BFX \ar[ddrr]^{\theta_X} \\
    &&& \textsc{\tiny (T1)} && \textsc{\tiny (T2)} \\
    1_{\mathcal D'} \ar[uurr]^{\beta_1} \ar[rr]_{\gamma_1}
    \ar[ddddrrrr]_{\gamma_1} && G\varnothing_{\mathcal C'} \ar[rr]_{G\alpha_1}
    && GA\varnothing_{\mathcal C} \ar[rrrr]_{GA!} &&&& GAX \\
    &&& \textsc{\tiny (A1)} && \textsc{\tiny (A2)} \\
    \\
    \\
    &&&& G\varnothing_{\mathcal C'} \ar[uuuu]^{\sim}_{G!} \ar[uuuurrrr]_{G!}
  }
\]
Here \textsc{(A1)} and \textsc{(A2)} commute by the fact $A$ preserves colimits,
\textsc{(T1)} commutes by the unit law for the monoidal natural transformation
$\theta$, and \textsc{(T2)} commutes by the naturality of $\theta$.

Thus
we have the equality of decorations $\theta_X \circ B(F! \circ \varphi_1) \circ
\beta_1
= G! \circ \gamma_1\colon  1 \to GAX$, and so $T$ sends identity morphisms to
identity morphisms.

\subsection{Monoidality of functors between decorated cospan categories.} \label{app:naturality2}

\noindent

The monoidality of a functor $(T,\tau)$ has two aspects: the naturality of the
transformation $\tau$, and the coherence axioms. We discuss the former; since
$\tau$ is just an empty-decorated version of $\alpha$, the latter then
immediately follow from the coherence of $\alpha$.

The naturality of $\tau$ may be proved via the same method as that employed in
Appendix \ref{app:naturality}; we first use Appendix \ref{app:emptydecorations}
to compute the two paths around the naturality square, and then use the
naturality of the coherence map $\alpha$ to show that these two decorated
cospans are isomorphic.

In slightly more detail, suppose we have decorated cospans
\[
  (X \stackrel{i_X}\longrightarrow N
\stackrel{o_Z}\longleftarrow Z,\enspace 1 \stackrel{s}\longrightarrow FN) \quad
\textrm{and} \quad (Y \stackrel{i_Y}\longrightarrow M
\stackrel{o_W}\longleftarrow W,\enspace 1 \stackrel{t}\longrightarrow FM).
\]
Then naturality demands that the cospans
\[
 AX+AY \stackrel{Ai_X+Ai_Y}{\xrightarrow{\hspace*{1.5cm}}} AN+AM
 \stackrel{(o_Z+o_W)\circ\alpha^{-1}}{\xleftarrow{\hspace*{1.5cm}}} A(Z+W)
\]
and
\[
 AX+AY \stackrel{A(i_X+i_Y)\circ\alpha}{\xrightarrow{\hspace*{1.5cm}}} A(N+M)
 \stackrel{A(o_Z+o_W)}{\xleftarrow{\hspace*{1.5cm}}} A(Z+W)
\]
are isomorphic as decorated cospans, with decorations the top and bottom rows of
the diagram below respectively.
\[
  \xymatrixrowsep{0.5pc}
  \xymatrix{
    & 1 \otimes 1 \ar[r]^(.4){\beta_1 \otimes \beta_1} & B1 \otimes B1 \ar[r]^(.4){Bs
    \otimes Bt} & BFN \otimes BFM \ar[r]^{\theta_N \otimes \theta_M} & GAN \otimes
    GAM \ar[r]^{\gamma_{AN,AM}} & G(AN+AM) \ar[dd]^{G\alpha_{N,M}} \\
    1 \ar[ur]^{\lambda^{-1}} \ar[dr]_{\beta_1} \\
    & B1 \ar[r]_(.4){B\lambda^{-1}} & B(1 \otimes 1) \ar[r]_(.4){B(s\otimes t)} & B(FN
    \otimes FM) \ar[r]_{B\varphi_{N,M}} & B(F(N+M))
    \ar[r]_{\theta_{N,M}} & G(A(N+M)).
  }
\]
As it is a subdiagram of the large commutative diagram in Appendix \ref{app:functors},
this diagram commutes. The diagrams required for $\alpha_{N,M}$ to be a morphism of
cospans also commute, so our decorated cospans are indeed isomorphic. This
proves $\tau$ is a natural tranfomation.

\end{document}